\documentclass[12pt,reqno]{article}
\usepackage[utf8]{inputenc}
\usepackage[letterpaper, margin=1in]{geometry}

\usepackage{amsmath}
\usepackage{amssymb}
\usepackage{amsthm}
\usepackage{amsfonts}
\usepackage{mathtools}

\usepackage{graphicx}

\usepackage{xcolor}
\usepackage{verbatim}
\usepackage{enumitem}
\usepackage{hyperref}

\usepackage[sorting=nty, maxbibnames=99]{biblatex}
\usepackage{float}
\usepackage{subcaption}
\usepackage{tikz}
\usepackage{forest}
\usepackage[ruled,vlined]{algorithm2e}

\theoremstyle{plain}
\newtheorem{theorem}{Theorem}
\newtheorem{corollary}[theorem]{Corollary}
\newtheorem{lemma}[theorem]{Lemma}
\newtheorem{prop}[theorem]{Proposition}

\theoremstyle{definition}
\newtheorem{definition}{Definition}
\newtheorem{example}[definition]{Example}

\theoremstyle{remark}
\newtheorem{remark}[definition]{Remark}

\definecolor{myblue}{HTML}{0021A5}
\definecolor{myorange}{HTML}{FA4616}

\title{Combinatorial Aspects of the Card Game War}
\author{Tanya Khovanova \and Atharva Pathak}
\date{}

\begin{document}

\maketitle

\begin{abstract}
This paper studies a single-suit version of the card game War on a finite deck of cards. There are varying methods of how players put the cards that they win back into their hands, but we primarily consider randomly putting the cards back and deterministically always putting the winning card before the losing card. The concept of a \emph{passthrough} is defined, which refers to a player playing through all cards in their hand from a particular point in the game. We consider games in which the second player wins during their first passthrough.
    
We introduce several combinatorial objects related to the game: game graphs, win-loss sequences, win-loss binary trees, and game posets. We show how these objects relate to each other. We enumerate states depending on the number of rounds and the number of passthroughs.
\end{abstract}

\section{Introduction} \label{introduction}

Card games are interesting to children and mathematicians alike. Games like poker and blackjack are fascinating applications of game theory, probability, and statistics. Bayer and Diaconis's research into card shuffling \cite{diaconisdovetail} has had practical applications in casinos, summarized in Diaconis's discussion with Haran \cite{numberphile}. There is also ongoing research in games where one player must sequentially guess cards in a deck given varying shuffling methods and varying information feedback regarding the guesses \cite{ciucunofeedback, diaconis2020card, diaconis2020guessing, liu2020card}.

This paper focuses on the card game War, for which some previous research exists. Haqq-Misra \cite{predictability} finds a linear regression on the chance of a player winning a game on a standard 52-card deck given two initial statistics: the difference in the  sum of card values for the two players and the number of rounds a player wins during the first 26 rounds. Lakshtanov and Roshchina \cite{finiteness} show that, with random putback of cards, the expected duration of the game is finite. Alexeev and Tsimerman \cite{alexeev2010analysis} investigate a game similar to War, where instead players randomly draw cards from their hands, and the player with the higher card in a round wins the lower card but discards the higher card. Ben-Naim and Krapivsky \cite{parityandruin} consider a similar stochastic model of War where cards are drawn randomly from the players' hands, but the winner of a round keeps both cards. They find the expected length of a game on $N$ cards is $\mathcal{O}(N^2)$ or $\mathcal{O}(N^2 \log N)$ depending on the initial advantage of the eventual winner. Spivey \cite{Spivey2010} investigates and classifies some types of cycles in games of War given deterministic putback.

We note that previous analysis in War generally applies to variations where the players randomly draw cards from their hands \cite{alexeev2010analysis, parityandruin} or are statistical analyses of large decks \cite{predictability}. We focus our analysis on standard rules of War in situations in which one player has many more cards than the opponent. In contrast to Spivey \cite{Spivey2010}, we focus on situations where one player wins, rather than games that cycle. 

Section~\ref{sec:preliminaries} defines the rules of War, the notation for how we represent games, and some additional definitions. We say the two players are Alice and Bob. We assume that our deck contains $n$ cards of one suit. A single-use game is a game where Bob wins while going through his initial hand. We only consider single-use games for the remainder of the paper. We define passthrough as a player going through the cards in their hand from a certain point in the game. In a single-use game, the game ends on Bob's first passtrough. The number of passthroughs for Alice is an important parameter in our calculations. We describe two ways that players may put cards they win back into their hand: WL-putback, where the winning card is put before the losing card, and random putback, where the order is chosen at random. 
Section \ref{sec:blocks_and_wlbt} introduces the game graph. The game graph has cards as vertices, and cards are connected if they play against each other. We prove that the game graph of a single-use game is a forest. We introduce the notion of blocks, which are irreducible components of the game.

In Section~\ref{sec:winlosssequences&bt} we introduce win-loss sequences that describe the results of every round and prove that the number of win-loss sequences corresponding to a single-use game is an entry in the Catalan triangle. We also prove a recursion for the number of win-loss sequences corresponding to games where Alice undergoes at most $k$ passthroughs. We describe a bijection between win-loss sequences corresponding to $k$-passthroughs and binary trees of height $k$. We show how to build the game graph given the win-loss sequence.

In Section~\ref{sec:countinggames} we find the probability that a randomly chosen initial state for a single-use game leads to a game with (1) exactly $R$ rounds or (2) at most $k$ passthroughs for Alice. We show that these probabilities are the same whether the game is played with random putback or WL-putback, which is a nontrivial result. 

In Section~\ref{sec:wlgamedigraphs} we discuss a way of giving direction to edges in the game graph such that it becomes a poset on the cards. Using these posets, we enumerate the number of initial states that necessarily follow a given win-loss sequence, given WL-putback. 

In Section~\ref{sec:randomputbackposet} we introduce posets related to random putback. We enumerate states that necessarily follow a given win-loss sequence with random putback by building a poset on the cards and counting its linear extensions.

\section{Preliminaries and Definitions} \label{sec:preliminaries}

The game of War is played between two players, whom we refer to as Alice and Bob. A deck of $n$ cards, labeled 1 through $n$, is divided between them, with their hands face down. A \emph{round} consists of both players revealing the top card from their hands, with the player whose card was higher collecting both cards to the bottom of their hand. The player who loses all their cards first loses the game. 

We consider two paradigms for how players put the cards they win in a round back to the bottom of their hand, which we call the \emph{putback}. Players may put the cards back \emph{randomly}, with Alice's card and Bob's card in either order, or deterministically. There are a few plausible ways of deterministically putting cards back, but we only consider \emph{WL-putback}, where the winning card is put back first and the losing card is put back second.

We represent a \emph{state} of the game by a string of the form $a_1 \dots a_i | a_{i+1} \dots a_n$, where $a_1 \dots a_i$ represents Alice's hand from top to bottom, $a_{i+1} \dots a_n$ represents Bob's hand from top to bottom, and the vertical bar character $|$ denotes the separation between the two players' hands. The values of these cards range from $1$ to $n$, as stated earlier. 

\begin{example}
Consider the game with initial state $2|13$ under random putback. Alice wins the first round because her card $2$ is greater than Bob's $1$. She randomly puts the cards back into her hand, so the new state is either $12|3$ or $21|3$. In both these cases, Bob wins the second round because he has the highest card of $3$, and then the possible states are $2|13, 2|31, 1|23, 1|32$. The first of these four cases is the initial state that we analyzed previously. In the latter three of the four cases, Bob wins the next round and, therefore, the game.
\end{example}

The player with the highest card can never lose in single-suit War. However, we do see that a loop is possible, with the states $12|3$ and $2|13$ potentially going to each other indefinitely long.

It is then immediate that there are $(n+1)!$ states of the game: each permutation of the bar and the numbers $1$ through $n$ represents a unique state of the game. This number includes states where the bar is at the beginning, corresponding to Bob winning, and states where the bar is at the end, corresponding to Alice winning.

We define an \emph{$m$-card state} as a state where Alice has $m$ cards. We refer to 1-card states as \emph{unicard states}. Note that there are $n!$ $m$-card states for every $m$.

The concept of a \emph{passthrough} is vital to our discussion throughout the paper. Given a state where a player has $m$ cards, we say that their hand undergoes a passthrough after $m$ rounds, and that these $m$ rounds occur during the passthrough. The concept of a passthrough encodes when a player's hand has been ``used up". We normally refer to passthroughs from the initial game state.

A \emph{single-use game} is a game in which Bob wins during his first passthrough. A game that ends within $k$ passthroughs for Alice is a \emph{$k$-passthrough game}. A game that ends in exactly $R$ rounds is a \emph{$R$-round game}. (Note the difference between ending \emph{within} $k$ passthroughs and ending in \emph{exactly} $R$ rounds.) 

For simplicity, we sometimes implicitly refer to the game played from some initial state as simply that state. For example, when we say a state is $R$-round, we mean the game initialized from that state is $R$-round.

We will need some standard definitions in combinatorics. To recall, a \textit{full binary tree} is a tree in which every vertex other than the leaves has two children.

The entry $C(n, k)$ of the Catalan triangle is defined as the number of strings of $n$ $U$'s and $k$ $D$'s such that each initial substring has at least as many $U$'s as $D$'s \cite{catalantriangle}. It is known that \[C(n, k) = \binom{n+k}{k} - \binom{n+k}{k-1} =  \frac{n+1-k}{n+1} \binom{n+k}{k}.\]
We use $C_r$ to denote the $r$th Catalan number $\frac{1}{r+1}\binom{2r}{r}$. The Catalan numbers correspond to the diagonal of the Catalan triangle: $C_r = C(r, r)$.

\section{Game Graphs and Blocks} \label{sec:blocks_and_wlbt}

We construct a \emph{game graph} as follows. We begin with vertices for each of the $n$ cards. As we watch the game progress, we add edges between cards that play each other. Game graphs can have multiple edges if a pair of cards played against each other multiple times, but cannot have loops.

\begin{prop}
The game graph of a single-use game is a forest with $m$ nontrivial trees and $x$ isolated vertices, where $m$ is the number of Alice's initial cards and $x$ is the number of Bob's cards that never get played.
\end{prop}

\begin{proof}
We initialize the game graph with all cards as nodes, and then we add edges corresponding to rounds. We color nodes for each of Alice's $m$ cards as amber and Bob's as blue. In every subsequent round, an amber card will play against a blue card, and we add an edge between them and recolor the blue card as amber. 
Therefore, amber connected components grow only by the addition of blue vertices, so they stay disconnected and grow as trees. Moreover, each of Alice's initial cards plays at least once, so they are all in nontrivial trees. Cards that never get played for Bob remain as isolated blue vertices. 
\end{proof}

For a single-use game, the game graph decomposes into connected components. Treating cards outside a particular component as invisible, the component behaves as its own independent sub-game of War. We call these subgames \emph{blocks}. The components that are not isolated vertices contain exactly one card from initial Alice's deck. This is why the unicard states play an important role: they are the building blocks of single-use games. 

There are two ways we can turn our game graph into a digraph. We can give direction to edges from the card that won the round to the card that lost the round, or we can give direction to edges from the card Alice played in the round to the card Bob played in the round. 

The first method turns the game graph into a poset, where greater cards point towards lesser cards. We call such posets \textit{winner-to-loser game digraphs}. These posets will be explored in greater depth in Section~\ref{sec:wlgamedigraphs}. 

We call the digraph resulting from the second method an \textit{Alice-to-Bob game digraph}. In an Alice-to-Bob game digraph, the $m$ directed trees have unique roots. These unique roots are Alice's $m$ initial cards, because the trees were built starting from her initial cards and adding edges pointing to cards Bob played. 

Consider what the game graph would have been if we started tracking rounds from some later point in the game. If we built the game graph starting from round $i$ instead of round 1, the effect would be to delete the edges corresponding to rounds 1 to $i-1$. In the Alice-to-Bob digraph starting from round $i$, connected components are trees whose roots are the cards Alice had at that point in the game. We refer to subgames involving cards and rounds in each of these trees as \emph{subblocks} of the original game. We say these subblocks are \emph{induced} by Alice's cards which are the roots of the trees. A card induces a subblock each time it plays, so we always specify the round from which the subblock we want to discuss is begun.

Note that although a unicard game consists of only one block, it immediately separates into two subblocks after a first round win for Alice. In particular, from the state $a|bcd\dots$, if Alice wins the first round and uses WL-putback, the state becomes $ab|cd\dots$, and from this point on $a$ and $b$ induce subblocks.

\section{Win-loss Sequences and Binary Trees}\label{sec:winlosssequences&bt}

\subsection{Win-loss sequences}

A \textit{win-loss sequence} is a string $x_1 x_2 \dots$ of $W$s and $L$s describing the progress of the game from the point of view of Alice. A $W$ represents a round Alice wins, and an $L$ represents a round Alice loses. We use $R$ to denote the total number of rounds, $w_i$ and $\ell_i$ to denote the numbers of $W$s and $L$s within the first $i$ rounds, and we say $w = w_R$ and $\ell = \ell_R$. As a reminder, we only consider games where Alice loses. 

We can ``stylize'' a win-loss sequence with forward slashes to separate passthroughs of Alice's hand. For example, if Alice started with a single card, the win-loss sequence $WLL$ would be stylized as $W/LL$ because the first round constitutes Alice's first passthrough.

The number of letters before the first slash is the number of Alice's cards at the start. This is the same as the number of rounds in the first passthrough. After that, the number of letters between slashes is twice the number of $W$'s in the previous passthrough. As we only consider finite games, the rounds in the final passthrough are all $L$'s. Since each $W$ nets $+1$ cards for Alice and each $L$ nets $-1$ cards for Alice, the number of cards she has after $i$ rounds is $m + w_i - \ell_i$. She must have a positive number of cards until the end of the game when she has none, so for $1 \leq i < R$ we have $m + w_i - \ell_i > 0$ and $m + w_R - \ell_R = m + w - \ell = 0$.

Note that $w_i + \ell_i = i$, so $w + \ell = R$, and combining this with $m + w - \ell = 0$ finds that $m$ and $R$ have the same parity. We also trivially have $R \geq m$.

We now enumerate the number of win-loss sequences that are exactly $R$ rounds.

\begin{theorem}\label{thm:catalantriangle}
The number of win-loss sequences corresponding to an $m$-card $R$-round game is the entry $C(\frac{m+R}{2} - 1, \frac{R-m}{2})$ of the Catalan triangle.
\end{theorem}

\begin{proof}
Every win-loss sequence ends with an $L$ because Alice must go from having one card to having no cards. Ignore this final $L$. Then the condition on the remaining win-loss sequence is that no terminal segment of the string can have more $W$s than $L$s, because that would mean Alice had zero cards at some point before the end of the game. This is equivalent to the condition on strings that Catalan triangle numbers count, with the roles of terminal and initial segments reversed. Since we have ignored the final $L$, the remaining number of $L$'s and $W$'s is $\ell-1$ and $w$ respectively, and thus there are $C(\ell-1, w) = C(\frac{m+R}{2} - 1, \frac{R-m}{2})$ win-loss sequences corresponding to an $m$-card $R$-round game.
\end{proof}

\begin{corollary}
There are $C_{\frac{R-1}{2}}$ win-loss sequences corresponding to unicard games ending in $R$ rounds.
\end{corollary}

\begin{proof}
Win-loss sequences corresponding to unicard games have $\ell - w = m = 1$, so substitution gives $C(\ell - 1, w) = C(w, w) = C_{w}$. Unicard games have  $w = \frac{R-1}{2}$, so $C_{w} = C_{\frac{R-1}{2}}$.
\end{proof}

Here, we enumerate the number of win-loss sequences that end within $k$ passthroughs of Alice's hand. 

\begin{theorem}
The number of win-loss sequences corresponding to an $m$-card $k$-passthrough game is $A_k^m$, where $A_k$ is recursively defined by $A_1 = 1$ and $A_{k+1} = 1 + A_k^2$. 
\end{theorem}

\begin{proof}
We show inductively that the number of $k$-passthrough win-loss sequences is $A_k$ for unicard blocks. Then, combining the $m$ blocks' win-loss sequences gives $A_k^m$ combined win-loss sequences. Combining the sequences occurs by concatenating each of the sequences' first passthroughs, then each of their second passthroughs, and so on.

There is only $A_1 = 1$ unicard block $1$-passthrough win-loss sequence, namely $L$ itself. Then, there are two cases for a $(k+1)$-passthrough win-loss sequence: either an immediate $L$, or a $W$ and a continuation to a subsequent passthrough. In the second case, the win-loss sequences for the subblocks induced by the two cards yielded back to Alice from the $W$ must end within $k$ passthroughs, so there are $A_k^2$ possibilities here. Thus there are $A_{k+1} = 1 + A_k^2$ $k$-passthrough win-loss sequences for unicard blocks, as desired.
\end{proof}

\subsection{Win-loss binary trees}

We describe a bijection between unicard games and full binary trees. This is easily extended to $m$-card games and forests of $m$ full binary trees, where each of the $m$ blocks corresponds to a full binary tree. (Recall that ``full'' means each node has either zero children, in which case it is a leaf, or two children, in which case it is a non-leaf.)

Note that the bijection to binary trees in the following theorem works for any win-loss sequence, not just those corresponding to single-use games. However, we do restrict ourselves to games that Alice loses because such games end with a passthrough of $L$s.

\begin{theorem}
Win-loss sequences corresponding to $k$-passthrough unicard games are in bijection with full binary trees of height $k$.
\end{theorem}

\begin{proof}
We describe a construction of the bijection by demonstrating the conversion from win-loss sequences to binary trees and vice versa. 

We begin with a root node for the tree at the top. Each round corresponds to a node in the tree, which we label with $W$ and $L$ for win or loss. Every $L$ yields no cards back to Alice's hand, so we make $L$ nodes terminal leaf nodes. Every $W$ yields two cards back to Alice, and these two cards are thereby associated with this particular $W$. We give every $W$ two children nodes, which we label $W$ or $L$ depending on whether the corresponding card goes on to win or lose its next round. Importantly, the first card put back corresponds to the left child node, and the second card put back corresponds to the right child node. Once Alice undergoes a passthrough, the cards available for the next passthrough are those won during the prior passthrough. Since we wrote in the left-to-right direction and because the order of cards in Alice's stack is preserved after being put back, each passthrough in the win-loss sequence read left-to-right is exactly a level of the tree read left-to-right.

To convert back from a binary tree to a win-loss sequence, we orient the tree with the root at the top and then write $W$ on all non-leaves and $L$ on all leaves. Then we read off the sequence left-to-right, top-to-bottom. Since the number of leaves is always less than or equal to the number of non-leaves until the very end of the tree, the associated win-loss sequence always has at least as many wins as losses until the end, which means that Alice will have a positive number of cards until the end. Therefore, the win-loss sequence is valid, and the bijection is established both ways.
\end{proof}

Since each round corresponds to a node in a win-loss binary tree, we interchangeably refer to rounds and nodes in the win-loss binary tree.

\begin{example}
The tree corresponding to the sequence $W/WW/LWWL/LLLL$ played from initial state $a|bcdefghijkl$ with WL-putback is in Figure~\ref{fig:bintree_ex}. Because the game is single-use, the cards Bob wins are regarded as discarded. In parentheses, we say which card from Alice plays against which card from Bob. Passthrough is abbreviated as PT.
\end{example}

\begin{figure}[ht!]
    \centering
    \scalebox{0.95}{
    \begin{tikzpicture}
        \node (1) at (0, 0) {$W(a/b)$};
        \node (2) at (-3, -1.3) {$W(a/c)$};
        \node (3) at (3, -1.3) {$W(b/d)$};
        \node (4) at (-4, -2.6) {$L(a/e)$};
        \node (5) at (-2, -2.6) {$W(c/f)$};
        \node (6) at (2, -2.6) {$W(b/g)$};
        \node (7) at (4, -2.6) {$L(d/h)$};
        \node (8) at (-3, -3.9) {$L(c/i)$};
        \node (9) at (-1, -3.9) {$L(f/j)$};
        \node (10) at (1, -3.9) {$L(b/k)$};
        \node (11) at (3, -3.9) {$L(g/l)$};
        \node (level1) at (-8, 0)  {Initial: $a|bcdefghijkl$};
        \node (level2) at (-8, -1.3) {After PT1: $ab|cdefghijkl$};
        \node (level3) at (-8, -2.6) {After PT2: $acbd|efghijkl$};
        \node (level4) at (-8, -3.9) {After PT3: $cfbg|ijkl$};
        \draw (9) -- (5) -- (2) -- (1) -- (3) -- (6) -- (10);
        \draw (4) -- (2);
        \draw (8) -- (5);
        \draw (11) -- (6);
        \draw (7) -- (3);
    \end{tikzpicture}
    }
    \caption{The tree for win-loss sequence $W/WW/LWWL/LLLL$}
    \label{fig:bintree_ex}
\end{figure}
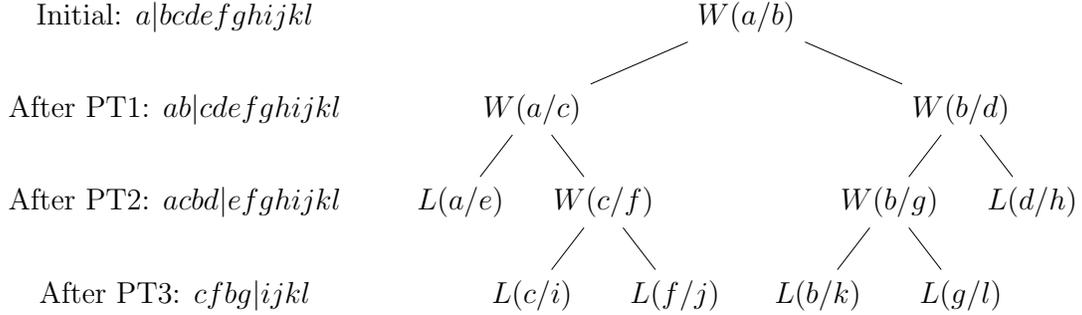

Now, we note that each subtree encodes a subblock. For example, take the subblock induced by card $a$ in the round $a/c$ in the second passthrough. The subblock plays as $a|cefij$ , which follows the win-loss sequence $W/LW/LL$, as shown by the subtree starting from the round $a$ against $c$ in the second level of the tree.

\subsection{Building the game graph from the win-loss tree}

We can create the game graph from a win-loss binary tree labeled with the two cards played in each round: we create a vertex for each card and an edge for each round. We want to show a formal algorithm to convert the win-loss binary tree into the game graph, but we need some definitions first.

Recall an ancestor of a node $x$ is a node on the path between the root and $x$, inclusive. We say that the \textit{right parent} of a node $x$ is the parent of the closest ancestor of $x$ which happens to be a right child. Note, however, that a node whose ancestors are all left children has no right parent. 

Consider a game with WL-putback. At each node of a win-loss binary tree, write the two cards that played in the corresponding round. We have the following lemma that describes the card Bob plays against through the win-loss binary tree.

\begin{lemma}\label{lemma:rightparent}
Consider a single-use $m$-card game played with WL-putback. Consider a round $Q$ where card $y$ plays for Bob against card $x$ for Alice. If $Q$ in a win-loss binary tree has right parent node $P$, then $x$ is the card Bob played in round $P$. If $Q$ does not have a right parent, then $x$ is the card from Alice's initial hand which induced the block that $Q$ is in.
\end{lemma}

\begin{proof}
After Alice wins a card $x$ from Bob in round $P$, it is put back second and hence plays in the right child round the next time it plays. Each time $x$ wins a subsequent round, it gets put back first due to WL-putback and hence plays in the left child round for the next time it plays. Eventually, $x$ plays $y$ in round $Q$. Working backwards to get from node $Q$ to node $P$ in the win-loss binary tree, we have to find the nearest ancestor of $Q$ which is a right child, corresponding to $x$ being put back second after being won from Bob, and then this nearest ancestor's parent, corresponding to the round $P$ where Bob played $x$.

A special case occurs if Alice never won card $x$ from Bob; i.e., when Alice began the game with $x$. In this case, $x$ is card $a$ from Alice's initial hand, and as long as $x$ continues winning, it gets put back first and next plays in the left child round. In this case, each round $x$ plays in has no right parent. Note that $x$ induces the block in which round $Q$ occurs. 
\end{proof}

Now we show the steps to obtain the game graph for a WL-putback game from the forest of win-loss binary trees labeled with the two cards played in each round.

\begin{enumerate}
    \item Add a node for each of Alice's initial cards above and to the left of the root of the tree they induced (so that the true root node is effectively a right child of this new node).
    \item Change the labels to the card played by Bob.
    \item Replace all edges with edges connecting each node to its right parent.
\end{enumerate}

We show the steps in Figure~\ref{fig:gettingggfromwlbt} for the unicard game starting from initial state $a|bcdef$ and following the win-loss sequence $W/LW/LL$.

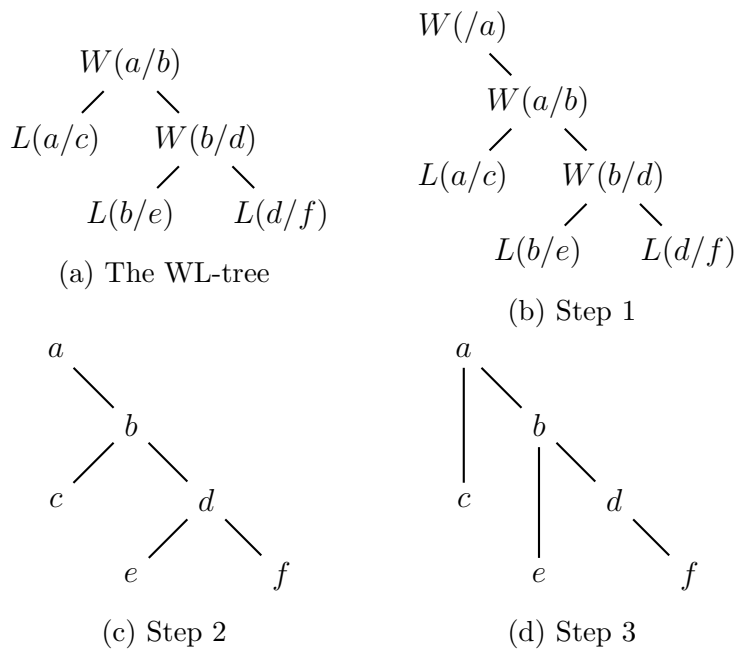
\begin{figure}[ht!]
    \centering
    \begin{subfigure}{0.32\textwidth}
        \centering
        \scalebox{1}{
        \begin{tikzpicture}
            \node (b) at (0, 0) {$W(a/b)$};
            \node (c) at (-1, -1) {$L(a/c)$};
            \node (d) at (1, -1) {$W(b/d)$};
            \node (e) at (0, -2) {$L(b/e)$};
            \node (f) at (2, -2) {$L(d/f)$};
            \draw[thick] (c) -- (b) -- (d) --(e);
            \draw[thick] (d) -- (f);
        \end{tikzpicture}}
        \caption{The WL-tree}
    \end{subfigure}
        \begin{subfigure}{0.32\textwidth}
        \centering
        \scalebox{1}{
        \begin{tikzpicture}
            \node (a) at (-1, 1) {$W(/a)$};
            \node (b) at (0, 0) {$W(a/b)$};
            \node (c) at (-1, -1) {$L(a/c)$};
            \node (d) at (1, -1) {$W(b/d)$};
            \node (e) at (0, -2) {$L(b/e)$};
            \node (f) at (2, -2) {$L(d/f)$};
            \draw[thick] (a) -- (b);
            \draw[thick] (c) -- (b) -- (d) --(e);
            \draw[thick] (d) -- (f);
        \end{tikzpicture}}
        \caption{Step 1}
    \end{subfigure}
    
    \begin{subfigure}{0.32\textwidth}
        \centering
        \scalebox{1}{
        \begin{tikzpicture}
            \node (a) at (-1, 1) {$a$};
            \node (b) at (0, 0) {$b$};
            \node (c) at (-1, -1) {$c$};
            \node (d) at (1, -1) {$d$};
            \node (e) at (0, -2) {$e$};
            \node (f) at (2, -2) {$f$};
            \draw[thick] (a) -- (b) -- (c);
            \draw[thick] (b) -- (d) -- (e);
            \draw[thick] (d) -- (f);
        \end{tikzpicture}}
        \caption{Step 2}
    \end{subfigure} 
    \begin{subfigure}{0.32\textwidth}
        \centering
        \scalebox{1}{
        \begin{tikzpicture}
            \node (a) at (-1, 1) {$a$};
            \node (b) at (0, 0) {$b$};
            \node (c) at (-1, -1) {$c$};
            \node (d) at (1, -1) {$d$};
            \node (e) at (0, -2) {$e$};
            \node (f) at (2, -2) {$f$};
            \draw[thick] (c) -- (a) -- (b);
            \draw[thick] (e) -- (b) -- (d) -- (f);
        \end{tikzpicture}}
        \caption{Step 3}
    \end{subfigure}
    \caption{Getting the game graph from win-loss binary tree with WL-putback} 
    \label{fig:gettingggfromwlbt}
\end{figure}

\begin{theorem}
The algorithm described above converts a forest of labeled win-loss binary trees into the game graph.
\end{theorem}

\begin{proof}
All second cards in the labeling of a win-loss tree, i.e.\ cards Bob plays in the game, are distinct due to the game being single-use. This is why step two gives us nodes with all possible cards in the game graph. Edges created in Step 3 describe pairs of cards that play against each other due to Lemma~\ref{lemma:rightparent}.
\end{proof}

We now motivate Step 1 of the algorithm using a unicard game as an example. Why does this trick of putting $a$ above and to the left of $b$ work? Let us say Bob somehow started with all the cards in the state $|abcd\dots$. Suppose Alice has a ``null'' card that beats $a$ and wins it over from Bob. Due to WL-putback, she puts $a$ second, which means it will next play in the round corresponding to the right child of this null versus $a$ round. From here on, the game operates normally from the state $a|bcd\dots$. In fact, we can treat such a null round as a parent of the true root of the win-loss binary tree.

\section{Counting \texorpdfstring{$R$}{R}-round and \texorpdfstring{$k$}{k}-passthrough Unicard Single-use Games}\label{sec:countinggames}

For a randomly initialized $m$-card game played with
\begin{enumerate}[label=(\alph*)]
    \item random putback
    \item or WL-putback,
\end{enumerate}
we can ask the probability the game ends up being
\begin{enumerate}
    \item $R$-round
    \item or $k$-passthrough.
\end{enumerate}

We only consider single-use games here. This means $n$, the number of cards in the deck, must be sufficiently large so that the number of cards Bob starts with, namely $n-m$, is at least the number of rounds.

At most how many rounds may take place in a $k$-passthrough game? During each of the first $k-1$ passthroughs, Alice can at most double her number of cards. By the end of the $k$th passthrough, she must have lost all her cards. So summing $m(1 + 2 + \dots + 2^{k-1})$ gives at most $m(2^k -1)$ total rounds.

\subsection{Counting \texorpdfstring{$R$}{R}-round single-use states} \label{sec:singleuse_exactlyrrounds}

In the following lemma we consider $m$-card states, where we assume that the total number of cards is $n$ and $0 < m < n$. Recall that the number of $m$-card states is $n!$, and likewise for $(m-1)$-card and $(m+1)$-card states.

\begin{lemma}\label{lem:uniform_probability}
Consider the uniform probability distribution over the $m$-card states. After one round is played from each of these states, half have Alice winning the first round and half have Bob winning the first round. Taking the half of cases where Alice wins, the probability distribution of the resulting states after random putback is the uniform distribution over all $(m+1)$-card states, and similarly taking the half of cases where Bob wins, the probability distribution of the resulting states is the uniform distribution over all $(m-1)$-card states. 
\end{lemma}

\begin{proof}
    Each $(m+1)$-card state is preceded by exactly one $m$-card state, which is the state obtained by reversing the process of a round that Alice won. This is done by removing the bottom two cards from Alice and putting the larger at the top of her hand and the smaller at the top of Bob's hand. Therefore, each $(m+1)$-card state is attainable from the $n!/2$ $m$-card states with a uniform probability distribution over the states. Likewise, we get the result for $(m-1)$-card states.
\end{proof}

This lemma allows us to prove the following theorem.

\begin{theorem}
The probability a random putback single-use game with randomly chosen $m$-card initial state is $R$-round, where $n \geq m+R$, is 
\begin{equation*}
    \frac{C(\frac{R+m}{2}-1, \frac{R-m}{2})}{2^R}.
\end{equation*}
\end{theorem}
\begin{proof}
The number of wins $w$ and the number of losses $\ell$ for an $m$-card $R$-round game must satisfy $\ell - w = m$ and $\ell + w = R$, so we have $\ell = \frac{R+m}{2}$ and $w = \frac{R-m}{2}$.

By Lemma~\ref{lem:uniform_probability}, we have that the probability distribution of the state after any number of rounds is the uniform distribution over states with the appropriate number of cards. Due to uniformity, the probability the card at the top of Alice's hand is greater than the card at the top of Bob's hand is $\frac{1}{2}$ at every point in the game. Thus, the chance Alice wins each round is $\frac{1}{2}$. By Theorem~\ref{thm:catalantriangle}, there are $C(\ell-1, w)$ win-loss sequences with $\ell$ and $w$ rounds of $L$ and $W$ respectively. The game could have followed any one of these with probability $1/2^R$ each, so the total probability of a game being $R$-round is $\frac{C(\ell-1, w)}{2^R}$, as claimed.
\end{proof}

For WL-putback, we do not have the same randomness after each round: states where Alice puts the lesser card before the greater card do not occur. For example, after Alice wins the first round of $3|124$, WL-putback makes the next state necessarily $31|24$ instead of $31|24$ or $13|24$ with equal probability. However, it turns out that we still have the same probability of a game with randomly chosen $m$-card initial state being $R$-round.

First, we prove two lemmas.

Note that for a unicard game, the system of equations are $\ell + w = R$ and $\ell - w = 1$. In particular, for the following lemma, we use $R = 2w + 1$.
\begin{lemma} \label{lem:unicard_wlputback_rround}
The probability that a randomly chosen single-use unicard state $a_1 | a_2 a_3 \dots a_n$ played with WL-putback is $R$-round is 
\begin{equation*}
    \frac{C_\frac{R-1}{2}}{2^R}.
\end{equation*}
\end{lemma}

\begin{proof}
We use strong induction on $w$. 

The base case $w = 0$ gives $\frac{C_w}{2^{2w+1}} = \frac{1}{2}$, which is indeed the probability that the game ends in exactly $2 \cdot 0+1=1$ round: this happens when $a_2$ beats $a_1$ in the first round itself.

For the inductive hypothesis, assume that the probability a randomly chosen single-use game is $(2x+1)$-round is $\frac{C_x}{2^{2x+1}}$ for every $x$ from $0$ to $r$. In order for the game to last exactly $2(w+1)+1 = (2w+3)$ rounds, Alice must first win the first round. So we must have $a_1 > a_2$, and then the state becomes $a_1 a_2 | a_3 \dots a_n$ after putting the cards back in the WL order. The $\frac{1}{2}$ chance that $a_1 > a_2$ is used again later.

However, temporarily assume that Alice had done random putback for this round before proceeding with WL-putback for the remainder of the game. Then the state is described as $b_1 b_2 | a_3 \dots a_n$, where $b_1$ and $b_2$ are $a_1$ and $a_2$ in some order. One round has taken place, so for the remainder of the game to last exactly in $2w+2$ total rounds, the subblocks induced by the next rounds of $b_1$ and $b_2$ must end in $2x+1$ and $2y+1$ rounds exactly, where $x$ and $y$ are numbers such that $(2x+1)+(2y+1) = 2w+2$ rounds.

Note that the state $a_1 a_2 | a_3 \dots a_n$ does not follow the uniform distribution over all $2$-card states, because we know $a_1 > a_2$, but the state $b_1 b_2 | a_3 \dots a_n$ does follow the uniform distribution over all $2$-card states, because we can have either $b_1 > b_2$ or $b_1 < b_2$. Therefore, the subblocks induced by $b_1$ and $b_2$ are totally independent. By the inductive hypothesis, the probability for these subblocks to end in $2x+1$ and $2y+1$ rounds is \[\frac{C_x}{2^{2x+1}} \frac{C_y}{2^{2y+1}} = \frac{C_x C_y}{2^{2w+2}}.\]
Let $A_x$ denote the probability that $a_1$'s subblock ends in $2x+1$ rounds and $a_2$'s subblock ends in $2w+2-(2x+1) = 2(w-x)+1$ rounds from the state $a_1 a_2 | a_3 \dots a_n$. Then $\frac{A_x + A_{w-x}}{2}$ is the average of two cases of
\begin{enumerate}
    \item $a_1$'s subblock ending in $2x+1$ rounds and $a_2$'s subblock ending in $2(w-x)+1$ rounds, and
    \item $a_1$'s subblock ending in $2(w-x)+1$ rounds and $a_2$'s subblock ending in $2x+1$ rounds.
\end{enumerate}
Importantly, this averaging of cases is exactly what we did with using $b_1$ and $b_2$ as $a_1$ and $a_2$ in either order. Therefore, we have
\[\frac{A_x + A_{w-x}}{2} = \frac{C_x C_{w-x}}{2^{2r+2}}.\]

The total probability of the game ending in $2w+3$ rounds is then
\begin{equation*}
    \frac{1}{2} \sum_{x=0}^{x=w} A_x,
\end{equation*}
where the first $\frac{1}{2}$ comes from the probability that $a_1$ beat $a_2$ initially. We have
\begin{align*}
    &\frac{1}{2} \sum_{x=0}^{x=w} A_x \\ 
    = &\frac{1}{2} \left( \frac{A_0 + A_w}{2} + \frac{A_1 + A_{w-1}}{2} + \dots + \frac{A_w + A_0}{2} \right)  \\
    = & \frac{1}{2} \sum_{x=0}^{x=w} \frac{C_x C_{w-x}}{2^{2w+2}}. \\
\end{align*}
Finishing with the Catalan recursive formula, we get
\begin{equation*}
\frac{1}{2}\sum_{x=0}^{x=w} \frac{C_x C_{w-x}}{2^{2w+2}} 
    = \frac{1}{2^{2w+3}} C_{w+1} = \frac{1}{2^{R+1}} C_{w+1}, 
\end{equation*}
and the induction is complete. 
\end{proof}

We need the following lemma about the properties of the Catalan triangle. 
\begin{lemma}\label{lem:sumprodcatalan}
We have
\begin{equation*}
    C(m+k-1, k) = \sum_{\substack{x_1 + x_2 + \dots + x_m = k \\ x_1, \dots, x_m \geq 0}} \prod_{i=1}^{i=m} C_{x_i}.
\end{equation*}
\end{lemma}

\begin{proof}
By definition $C(p, q)$ is the number of paths from $(0, 0)$ to $(p, q)$ consisting of up ($U$) and right $(R)$ moves that never go above the line $x-y=0$. 

We describe a bijection between
\begin{itemize}
    \item Type $A$: paths from $(0, 0)$ to $(m+k, k)$ consisting of $m+k$ copies of $R$ and $k$ copies of $U$ that never go above the line $x-y=1$ except before the first move to the right.
    \item Type $B$: sequences of $m$ subpaths that go from $(0, 0)$ to $(x_i+1, x_i)$ consisting of $x_i$ copies of $U$ and $x_i + 1$ copies of $R$ that never go above the line $x-y=1$ except before the first move to the right, such that $x_1 + \dots + x_m = k$ and $x_1, \dots, x_m \geq 0$. 
\end{itemize}
Let $\mathcal{A}$ and $\mathcal{B}$ denote the numbers of Type $A$ paths and Type $B$ sequences.

Type $A$ paths are equivalent to paths from $(0, 0)$ to $(m+k-1, k)$ but shifted one unit to the right and with an extra $R$ attached to the beginning. Therefore, we have
\begin{equation*}
    \mathcal{A} = C(m+k-1, k).    
\end{equation*}

The regular Catalan numbers $C_k$ count paths from $(0, 0)$ to $(k, k)$ that never go above $x-y=0$. Our second quantity $\mathcal{B}$ exactly counts sequences of such paths, again with an extra $R$ attached at the beginning of each path. Therefore, $\mathcal{B}$ is the sum over all possible values of $x_1, \dots, x_m$ of the product of the amounts of possibilities for the first through $m$th path. Then 
\begin{equation*}
    \mathcal{B} = \sum_{\substack{x_1 + x_2 + \dots + x_m = k \\ x_1, \dots, x_m \geq 0}} \prod_{i=1}^{i=m} C_{x_i}.
\end{equation*}
Once we establish the desired bijection, we will have shown $\mathcal{A} = \mathcal{B}$, as desired.

First, we describe how to convert from a Type $A$ path to a Type $B$ sequence. The path must touch the lines $x-y=i$ a last time for every $i$ from $1$ to $m-1$, because the path goes from an initial point $(0, 0)$ for which $x-y = 0-0 = 0$ to a final point $(m+k, k)$ for which $x-y = m+k - k = m$. Say that these points are $(x_1, y_1), \dots, (x_{m-1}, y_{m-1})$. Splitting the path at each of these points, we get a sequence of $m$ subpaths. Since $(x_i, y_i)$ is on the line $x-y = i$, the point can be expressed as $(y_i + i, y_i)$. So the subpath from $(y_i + i, y_i)$ to $(y_{i+1} + i+1, y_{i+1})$ has $(y_{i+1} - y_i)$ copies of $U$ and $(y_{i+1} - y_i + 1)$ copies of $R$. Moreover, the subpath does not go above the line $x-y = i+1$ except the final $R$, so each of the subpaths along this sequence of $m$ subpaths fits the description of Type $B$.

Now, we describe how to convert from a Type $B$ sequence to a Type $A$ path. We simply concatenate the $m$ subpaths into a full path. Each subpath contributes $x_i$ copies of $U$ and $x_i + 1$ copies of $R$, so in total there are $x_1 + \dots + x_m = k$ copies of $U$ and $(x_1 + 1) + \dots + (x_m + 1) = k+m$ copies of $R$. This is exactly how many $U$ and $R$ moves are necessary for a Type $A$ path from $(0, 0)$ to $(m+k, k)$. Further, this concatenated path cannot go above the line $x-y=1$ except at the starting point because each of the $m$ subpaths increase the value of $x-y$ by $1$ from an initial value of $0-0=0$, and the first subpath only has $x-y$ value of $0$ before the very first $R$.

Thus the bijection is established both ways, and $\mathcal{A} = \mathcal{B}$ as desired.
\end{proof}

We can now combine Lemma~\ref{lem:unicard_wlputback_rround} and Lemma~\ref{lem:sumprodcatalan} into Theorem~\ref{thm:mcard_wlputback_rround}.

\begin{theorem}\label{thm:mcard_wlputback_rround}
The probability that an $m$-card game played with WL-putback is $R$-round is 
\begin{equation*}
 \frac{1}{2^R}C\left(\frac{R+m}{2}-1, \frac{R-m}{2}\right).
\end{equation*}
\end{theorem}

\begin{proof}
In this proof, we use $w_i$, $\ell_i$, and $R_i$ for the numbers of wins, losses, and rounds in the $i$th block, instead of the numbers of such quantities within the first $i$ rounds. We have $R_i = \ell_i + w_i$, and since blocks are unicard, we have $\ell_i = w_i + 1$. Adding the wins and losses of all blocks, we have $w = w_1 + \dots + w_m$, $\ell = \ell_1 + \dots + \ell_m$, and $R = R_1 + \dots + R_m$. 

By Lemma~\ref{lem:unicard_wlputback_rround}, there is a $\frac{C_{w_i}}{2^{R_i}}$ probability that a particular block consists of exactly $w_i$ wins. Blocks are independent, so the probability of the $m$ blocks having $w_1, w_2, \dots, w_m$ wins is 
\begin{equation*}
    \prod_{i=1}^{i=m} \frac{C_{w_i}}{2^{R_i}} = \frac{1}{2^{R_1 + \dots + R_m}} \prod_{i=1}^{i=m} C_{w_i} = \frac{1}{2^R} \prod_{i=1}^{i=m} C_{w_i}.
\end{equation*}
This must be added over all possible assignments of numbers of wins to the blocks, so we have a total probability of 
\begin{equation*}
     \sum_{w_1 + \dots + w_m = w} \left( \frac{1}{2^R} \prod_{i=1}^{i=m} C_{w_i} \right) = \frac{1}{2^R} \sum_{w_1 + \dots + w_m = w} \prod_{i=1}^{i=m} C_{w_i}.
\end{equation*}
Applying Lemma~\ref{lem:sumprodcatalan}, this simplifies as
\begin{equation*}
    \frac{1}{2^R} \sum_{w_1 + \dots + w_m = w} \prod_{i=1}^{i=m} C_{w_i}  = \frac{1}{2^R} C(m+w-1, w).
\end{equation*}
We have $\ell - w = m$ and $\ell + w = R$, so we have $w = \frac{R-m}{2}$. Thus, 
\begin{equation*}
    \frac{1}{2^R} C(m+w-1, w) = \frac{1}{2^R}C\left(\frac{R+m}{2}-1, \frac{R-m}{2}\right),
\end{equation*}
as claimed.
\end{proof}

\subsection{Counting \texorpdfstring{$k$}{k}-passthrough single-use games}
Previously, our main parameter was the number of rounds. Now, we look at the number of passthroughs. 

In order for every possible $k$-passthrough progression of the game starting from an $m$-card state to result in Alice losing during Bob's first passthrough, $n$ must be sufficiently large. The most rounds that can occur within $k$ passthroughs of Alice is $2^{i-1}m$ for the $i$th passthrough, so summing from $i=1$ to $i=k$, this is $(2^k - 1)m$ rounds. Bob must have at least these many cards, so we assume $n \geq 2^km$ for this subsection.

\begin{theorem}\label{thm:counting_kpassthroughs}
The probability that a game initialized with $m$-card initial state is $k$-pass\-through single-use is $P_k^m$, where $P_k$ is recursively defined by $P_1 = \frac{1}{2}$ and $P_{k+1} = \frac{1}{2} + \frac{1}{2}P_k^2$, both for when
\begin{enumerate}[label=(\alph*)]
    \item the game is played with random putback
    \item and when the game is played with WL-putback.
\end{enumerate}
\end{theorem}

\begin{proof}
We first show using induction that the claimed probability is correct for unicard games (i.e., when $m=1$). Then, since blocks are independent, we simply multiply the $P_k$ probability for the $m$ blocks to obtain $P_k^m$. Our base case for both random and WL-putback is the $\frac{1}{2}$ chance that a unicard game $a_1|a_2 a_3\dots$ ends on the first passthrough, which occurs when $a_2 > a_1$. 
\begin{enumerate}[label=(\alph*)]
    \item Assume that the probability a unicard block played with random putback is $P_k$. There are two cases we consider for a $(k+1)$-passthrough game: we may have $a_2 > a_1$, so that the game ends in the first passthrough itself, or else we have $a_1 > a_2$, so that we continue to subsequent passthroughs. The first case occurs with probability $\frac{1}{2}$. In the second case, let $b_1$ be the first card put back and let $b_2$ be the second card put back among $a_1$ and $a_2$, so that the state can be written as $b_1 b_2 | a_3 \dots a_n$. This state is now a random permutation of the $n$ cards where Alice has 2 cards. The blocks induced by $a_1$ and $a_2$ are not independent, because we know $a_1 > a_2$. However, $b_1$ and $b_2$ can have either $b_1 > b_2$ or $b_2 > b_1$, so they induce their own, totally independent subblocks. For the original game to be $(k+1)$-passthrough, both of these subblocks must end within $k$ passthroughs. Due to independence and the inductive hypothesis, this occurs with probability $P_k^2$. Combining the $a_2 > a_1$ and $a_1 > a_2$ cases, we have
    \begin{equation*}
        P_{k+1} = \frac{1}{2} + \frac{1}{2} P_k^2,
    \end{equation*}
    as desired.
    \item We follow a strategy very similar to the random putback case. The same base case for 1-passthrough games applies. We wish to show that the probability the remainder of the game from the state $a_1 a_2 | a_3 \dots a_n$ that results from Alice winning the first round ends within $k$ passthroughs is still $P_k^2$, despite the non-independence of the blocks induced by $a_1$ and $a_2$. We begin with the inductive hypothesis that the probability a WL-putback game is $k$-passthrough is $P_k$. Then, let us consider a situation where Alice had done random putback instead of WL-putback for just the first round, so that the state is $b_1 b_2 | a_3 \dots a_n$ where $b_1$ and $b_2$ are $a_1$ and $a_2$ in some order. The subblocks induced by $b_1$ and $b_2$ are independent, so using the inductive hypothesis of the probability of a random initial state being $P_k$, we have that the probability the state $b_1 b_2 | a_3 \dots a_n$ is $(k+1)$-passthrough is $P_k^2$. By symmetry, this probability applies for both possible ways Alice could have put back her cards, and in particular the WL-putback case. Therefore, in the $a_1 > a_2$ case we have a probability $P_k^2$ that the initial game is $k+1$ passthrough. Combining the $a_2 > a_1$ and $a_1 > a_2$ cases as before, we obtain the same recursion, as claimed.
\end{enumerate}
\end{proof}

The first few values in the sequence $P_k$ are $P_2 = \frac{5}{8}$ and $P_3 = \frac{89}{128}$.

\begin{example}
We show the possible permutations for a 2-passthrough unicard game played with WL-putback. Such an initial state has form $a|bcd$. Alice may lose in the first pass\-through, in which case the only constraint is $a<b$. Otherwise, if Alice wins the first round, then Bob must win from the state $ab|cd$. In this case, the constraints are $a>b$, $c>a$, and $d>b$. Out of the 24 initial unicard states on 4 cards, 12 satisfy $a < b$, and additionally $2|143$, $2|134$, and $3|142$ satisfy $a>b$, $c>a$, and $d>b$. Thus, there is a $\frac{15}{24} = \frac{5}{8}$ chance that a game played with WL-putback initialized with random unicard state $a|bcd$ results in Bob winning within two passthroughs of Alice's hand.
\end{example}

\begin{example}
Now, we compute the probability that a game randomly initialized from a state of the form $a|bcd$ and played with random putback is 2-passthrough. The 12 initial states where $a<b$ result in Alice losing immediately on the first round with probability 1. The states $2|134$ and $2|143$ also result in a 2-passthrough game with probability 1, because after she wins the first round and randomly puts her cards back, Bob's remaining cards of $3$ and $4$ beat Alice's cards of $1$ and $2$. For the initial states $3|124$ and $3|142$, the game is only 2-passthrough with the $\frac{1}{2}$ probability depending on how she puts her cards back. The combined probability is $\frac{12 \cdot 1 + 2 \cdot 1 + 2 \cdot \frac{1}{2}}{24} = \frac{15}{24}= \frac{5}{8}$.
\end{example}

\begin{corollary}
The probability a unicard state is $k$-passthrough tends to 1 as $k$ tends to infinity.
\end{corollary}

\begin{proof}
Multiplying both sides of the recurrence relation $P_{k+1} = \frac{1}{2} + \frac{1}{2}P_k^2$ by 2 and subtracting $2P_k$ from both sides yields $2P_{k+1} - 2P_k = 1+P_k^2-2P_k$, or $2(P_{k+1} - P_k) = (1-P_k)^2$. The only steady-states of this recurrence are $P = 0$ and $P=1$, and the sequence starting at $P_1 = \frac{1}{2}$ is monotone increasing. Therefore, the sequence converges to the steady-state $P_{\infty} = 1$. 
\end{proof}

This is a nontrivial result. It was plausible that the chance Alice does not lose within Bob's first passthrough would approach some positive value if high-value initial cards for Alice ``survived'' long enough, but this is not the case.

\section{Winner-to-loser Game Digraphs}\label{sec:wlgamedigraphs}

We can take a game graph and make it a directed graph where each edge points from the card that won to the card that lost a round. In fact, since this is an acyclic directed graph, it constitutes a poset on the cards. Posets can be graphically represented by Hasse diagrams, where greater elements are drawn above lesser elements.

\begin{example}
Consider the game with WL-putback and the win-loss sequence $W/LW/LL$ with initial state $a|bcdef$. The game progresses as
\begin{alignat*}{2}
    &&  a&|bcdef \\
    &\implies& ab&|cdef \\
    &\implies& b&|defca \\
    &\implies& bd&|efca \\
    &\implies& d&|fcaeb \\ 
    &\implies& &|caebfd.
\end{alignat*}
The directed game graph, i.e.\ the poset, is shown in Figure~\ref{fig:wlwll_gameposet}.

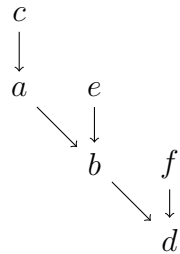
\begin{figure}[ht!]
    \centering
    \begin{tikzpicture}
    \node (a) at (0, 0) {$a$};
    \node (b) at (1, -1) {$b$};
    \node (c) at (0, 1) {$c$};
    \node (d) at (2, -2) {$d$};
    \node (e) at (1, 0) {$e$};
    \node (f) at (2, -1) {$f$};
    
    \draw[->] (b) -- (d);
    \draw[->] (a) -- (b);
    \draw[->] (c) -- (a);
    \draw[->] (e) -- (b);
    \draw[->] (f) -- (d);
    \end{tikzpicture}
    \caption{Game poset for WL-putback on $W/LW/LL$}
    \label{fig:wlwll_gameposet}
\end{figure}
\end{example}

\subsection{Poset for WL-putback}\label{sec:poset_wlputback}

When building each component of the game graph while watching the game progress, each component became a tree. Let us consider adding direction to build the winner-to-loser digraph. The non-leaves in each connected component form a tree poset, because these are cards/rounds that Alice wins from Bob. Leaves in each component are cards/rounds that Bob won with, so the edge direction points opposite to the tree of non-leaves. 

\begin{example}
Considering the win-loss sequence $W/WW/LWWW/LLLLLL$ with the initial state $a|bcdefghijklmn$, we build the winner-to-loser game digraph. The result is shown in Figure~\ref{fig:wwwlwwwllllll_poset}.

\begin{figure}[ht!]
    \centering
    \begin{tikzpicture}
    \node (a) at (0, 0) {$a$};
    \node (b) at (1, -1) {$b$};
    \node (c) at (-1, -1) {$c$};
    \node (d) at (2, -2) {$d$};
    \node (e) at (-1, 1) {$e$};
    \node (f) at (-1, -2) {$f$};
    \node (g) at (0, -2) {$g$};
    \node (h) at (2, -3) {$h$};
    \node (i) at (-2, 0) {$i$};
    \node (j) at (-2, -1) {$j$};
    \node (k) at (2, 0) {$k$};
    \node (l) at (0, -1) {$l$};
    \node (m) at (3, -1) {$m$};
    \node (n) at (3, -2) {$n$};
        
    \draw[thick, dashed, ->] (e) -- (a);
    \draw[thick, dashed, ->] (k) -- (b);
    \draw[thick, dashed, ->] (i) -- (c);
    \draw[thick, dashed, ->] (m) -- (d);
    \draw[thick, dashed, ->] (j) -- (f);
    \draw[thick, dashed, ->] (l) -- (g);
    \draw[thick, dashed, ->] (n) -- (h);
    \draw[ultra thick, ->] (a) -- (b);
    \draw[ultra thick, ->] (b) -- (d);
    \draw[ultra thick, ->] (d) -- (h);
    \draw[ultra thick, ->] (a) -- (c);
    \draw[ultra thick, ->] (c) -- (f);
    \draw[ultra thick, ->] (b) -- (g);
    \end{tikzpicture}
    \caption{Game poset for WL-putback on $W/WW/LWWW/LLLLLL$}
    \label{fig:wwwlwwwllllll_poset}
\end{figure}
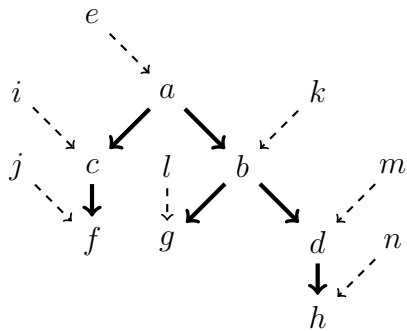

In this Hasse diagram, we have drawn thin dashed edges to connect a card from Bob that beats a card from Alice. Taken alone, the solid edges and the cards they connect form a tree poset structure rooted at $a$. Nodes of this tree poset are the cards that have been in Alice's hand at any point in the game. In the Hasse diagram, each such node has one node above it corresponding to the card from Bob's hand that beat it. These cards from Bob's hand sit outside the tree.
\end{example}

For non-unicard games, the game digraph consists of $m$ connected components, and for each connected component the non-leaves form a tree poset. 

We want to show how to count the number of states that satisfy a win-loss sequence with $n-1$ rounds given WL-putback. But first we need some preliminary discussion.

\subsection{Tree posets}\label{sec:treeposets}
Any tree inherits a natural poset structure, where the root is the largest element. We call this poset a \emph{tree poset}.

A \emph{linear extension} of a finite poset of $n$ elements is an assignment of the integers $1$ through $n$ such that poset relations are satisfied. In the context of War, counting linear extensions corresponds to the number of initial states that satisfy the constraints of a poset on the cards, which is our aim in this section. An element of a poset is said to \emph{cover} a lesser element if there is no element between them.

Let $h(v)$ denote the number of elements less than or equal to $v$ in a tree poset. Then we have the following lemma.

\begin{lemma}[Ruskey \cite{RUSKEY199277}]\label{lem:treeposet_counting}
The number of linear extensions of a tree poset with $n$ elements is 
\begin{equation*}       \frac{n!}{\prod_{i=1}^n h(v_i)}.
\end{equation*}
\end{lemma}
Note that this lemma is similar to the hook-length formula. A quick understanding of this formula arises from the fact that for each $v$, the probability it is the greatest in the subtree where $v$ is the root is $\frac{1}{h(v)}$. Multiplying these probabilities for every vertex, we get the expression in the lemma.

\subsection{Enumerating states with the given WL-sequence}

We use the variable $k$ with $n = 2k$, so that there are $k$ rounds of $L$ and $k-1$ rounds of $W$. Note that each card that ever appears in Alice's hand loses exactly once.

We now finish with a theorem to count linear extensions of the poset, which enumerates initial states that satisfy the win-loss sequence.

\begin{theorem}\label{thm:ns4wlsequence}
The number of $2k$-round single-use WL-putback unicard states $a_1 | a_2 \dots a_{2k}$ that satisfy a particular win-loss sequence  is
\begin{equation*}
    \frac{(2k)!}{2^k \prod_{i=1}^k h(v_i)},
\end{equation*}
where $v_i$'s are vertices in the tree component of the associated poset and $h(v_i)$ is the number of vertices less than or equal to $v_i$ in the tree component of the poset.
\end{theorem}

\begin{proof}
There are $\frac{(2k)!}{k! \cdot 2^k}$ partitionings of $2k$ cards into $k$ pairs, where the lesser element is in the tree component and the greater element sticks outside the tree. As described earlier, the number of linear extensions of just the tree component of the poset is
\begin{equation*}
    \frac{k!}{\prod_{i=1}^k h(v_i)}.
\end{equation*}
Multiplying these terms, we get
\begin{equation*}
    \frac{(2k)!}{k! \cdot 2^k} \cdot \frac{k!}{\prod_{i=1}^k h(v_i)} = \frac{(2k)!}{2^k\prod_{i=1}^k h(v_i)},
\end{equation*}
as claimed.
\end{proof}

We illustrate the theorem with an example. 

\begin{example} 
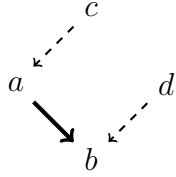
\begin{figure}[ht!]
    \centering
    \begin{tikzpicture}
    \node (a) at (0, 0) {$a$};
    \node (b) at (1, -1) {$b$};
    \node (c) at (1, 1) {$c$};
    \node (d) at (2, 0) {$d$};
    \draw[ultra thick, ->] (a) -- (b);
    \draw[thick, dashed, ->] (c) -- (a);
    \draw[thick, dashed, ->] (d) -- (b);
    \end{tikzpicture}
    \caption{Game poset for WL-putback on $W/LL$}
    \label{fig:wll_gameposet}
\end{figure}

Consider the game with initial state $a|bcd$ with win-loss sequence $W/LL$, played with WL-putback. We have four cards, so $k=2$. The game poset is shown in Figure~\ref{fig:wll_gameposet}. Within the tree component of the poset, consisting of $a$ and $b$, we have $h(a) = 2$ and $h(b) = 1$. Plugging these values into formula in Theorem~\ref{thm:ns4wlsequence}, we have 
\begin{equation*}
    \frac{4!}{2^2 \cdot 2 \cdot 1} = \frac{24}{8} = 3.
\end{equation*}
These three initial states are $2|134$, $2|143$, and $3|142$.
\end{example} 

\begin{remark}
We can enumerate $m$-card states with a given win-loss sequence. Following the win-loss sequence breaks the game up into $m$ unicard blocks. We apply Theorem~\ref{thm:ns4wlsequence} to each of the $m$ blocks and multiply the possibilities together. 
\end{remark}

\section{Blocks and Random Putback}\label{sec:randomputbackposet}

In this section, we want to count unicard states of the form $a_1 | a_2 \dots a_n$ that \emph{necessarily} follow a certain win-loss sequence under random putback. We first describe how to construct the poset associated with the win-loss sequence, and then show how to count its linear extensions. When dealing with an $m$-card rather than unicard state, the claimed method of constructing the poset applies to each unicard block, after which we can simply multiply the possibilities for each block.

It is important to note the word ``necessarily'' in a state \emph{necessarily} following a win-loss sequence. A given initial state can follow multiple possible win-loss sequences given different possibilities in random putback. For example, take the initial state $3|14256$. Alice wins the first round, and the state becomes either $31|4256$ or $13|4256$. In the first case, Alice loses the next two rounds, so the win-loss sequence is $W/LL$. In the second case, Alice loses and then wins, so that she has the cards $3$ and $2$ in some order. Both of these cards are less than both of $5$ and $6$, so Alice loses the next two rounds. This case follows the win-loss sequence $W/LW/LL$. Therefore, this initial state $3|14256$ does not necessarily follow any win-loss sequence.

\subsection{Building the random putback poset}

\begin{example}
Suppose an initial state $a|bcd$ is to necessarily follow win-loss sequence $W/LL$. Then, necessarily $a > b$ to get to the state $ab|cd$ or $ba|cd$, after which both $a$ and $b$ are less than $c$ and $d$ to necessarily have two $L$'s. The only initial states that satisfy these constraints are $2|143$ and $2|134$. The relationships are summarized in the poset in Figure~\ref{fig:wll_poset}.

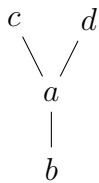
\begin{figure}[ht!]
    \centering
    \begin{tikzpicture}
    \node (a) at (0, 0) {$a$};
    \node (b) at (0, -1) {$b$};
    \node (c) at (-0.5, 1) {$c$};
    \node (d) at (0.5, 1) {$d$};
    \draw (c) -- (a) -- (b);
    \draw (d) -- (a);
    \end{tikzpicture}
    \caption{Poset for random putback on $W/LL$}
    \label{fig:wll_poset}
\end{figure}
\end{example}

Now, we describe the general method of obtaining the poset.
\begin{enumerate}
    \item Initialize the poset with a node for card $a$.
    \item At each node of a win-loss binary tree, write down which of Bob's cards plays against Alice in the round.
    \item Prune all leaves of the win-loss binary tree and put each corresponding card from the leaves as a node above $a$ in the Hasse diagram.
    \item Take the remainder of the tree, consisting of Bob's cards involved in $W$ rounds, and place the entire structure under $a$ in the Hasse diagram.
\end{enumerate}

The algorithm is  illustrated in Figure~\ref{fig:wlbt_poset_wwwlwwwllllll} using the example of the win-loss sequence $W/WW/LWWW/LLLLLL$ on the initial state $a|bcd\dots mn$. Card $a$ initializes the poset. Then, cards $e$, $i$, $j$, $k$, $l$, $m$, and $n$, which participated in $L$ rounds, are placed above $a$. The remainder of the binary tree, consisting of $b$, $c$, $d$, $f$, $g$, and $h$, is placed under $a$.

\begin{figure}[ht!]
    \centering
    \begin{subfigure}{0.63\textwidth}
        \scalebox{0.8}{
        \begin{forest}
            [$W(b)$
            [$W(c)$
            [$L(e)$]
            [$W(f)$[$L(i)$][$L(j)$]]
            ]
            [$W(d)$
            [$W(g)$[$L(k)$][$L(l)$]]
            [$W(h)$[$L(m)$][$L(n)$]]
            ]
            ]
        \end{forest}}
    \end{subfigure}
    \begin{subfigure}{0.36\textwidth}
        \scalebox{0.8}{
        \begin{tikzpicture}
            \node (a) at (0, 0) {$a$};
            \node (e) at (-3, 1) {$e$};
            \node (i) at (-2, 1) {$i$};
            \node (j) at (-1, 1) {$j$};
            \node (k) at (0, 1) {$k$};
            \node (l) at (1, 1) {$l$};
            \node (m) at (2, 1) {$m$};
            \node (n) at (3, 1) {$n$};
            \node (b) at (0, -1) {$b$};
            \node (c) at (-0.5, -2) {$c$};
            \node (d) at (1, -2) {$d$};
            \node (f) at (-0.5, -3) {$f$};
            \node (g) at (0.5, -3) {$g$};
            \node (h) at (1.5, -3) {$h$};
            
            \draw (m) -- (a) -- (n);
            \draw (e) -- (a);
            \draw (i) -- (a) -- (j);
            \draw (k) -- (a) -- (l);
            \draw (k) -- (a) -- (l);
            \draw (a) -- (b) -- (c) -- (f);
            \draw (b) -- (d) -- (g);
            \draw (d) -- (h);
        \end{tikzpicture}}
    \end{subfigure}
    \caption{Win-loss binary tree and poset with random putback for necessarily following $W/WW/LWWW/LLLLLL$}
    \label{fig:wlbt_poset_wwwlwwwllllll}
\end{figure}
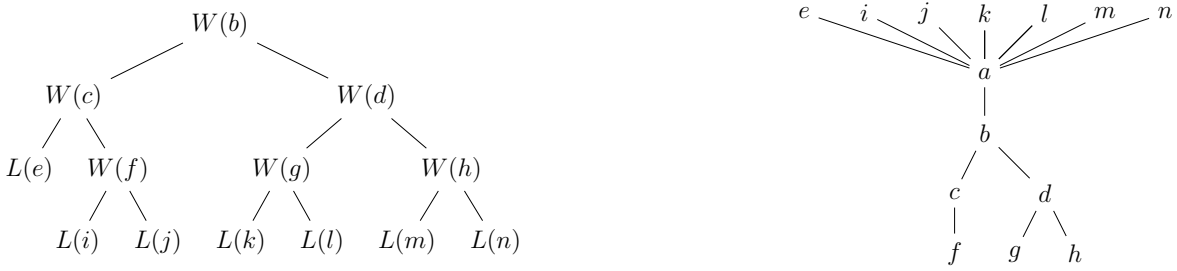

Now, we explain why the algorithm for writing the Hasse diagram works. Take a $W$ round where Alice's card $x$ beats Bob's card $y$. Each such $y$ then appears only in the subblock induced by that round. This subblock is exactly the subtree rooted at that round in the win-loss binary tree. This $y$ can appear anywhere in the subblock until the end of the subblock due to random putback; if put back first, $y$ begins a subblock corresponding to the left child subtree of the $x/y$ round; if put back second, $y$ begins a subblock corresponding to the right child subtree of the $x/y$ round. Therefore, every one of Bob's cards involved in subsequent $W$ in the subblock induced by $x/y$ must be less than $y$ for the $W$ to necessarily occur. This continues recursively through all $W$'s: each one of Bob's cards involved in a $W$ must be less than each of Bob's cards involved in a $W$ higher up in the tree. Moreover, card $a$ must be greater than every card played by Bob in a $W$ round, because $a$ can be anywhere in Alice's hand at every point in the game. This is exactly what we do when taking the structure of $W$'s in the binary tree and placing it under $a$ in the Hasse diagram.

Again, card $a$ can be anywhere in Alice's stack at every point in the game. Therefore, for $L$'s to occur necessarily, cards Bob plays in $L$ rounds must be greater than $a$. This is what we do when pruning $L$'s and putting the card Bob plays in those rounds above $a$ in the Hasse diagram.

\subsection{Enumerating states}

We now enumerate states that are guaranteed to follow a particular win-loss sequence under random putback by counting linear extensions of the poset. 

\begin{theorem}
The number of single-use random putback games with initial unicard state $a_1 | a_2 \dots a_{2k}$ that necessarily satisfy a particular win-loss sequence is 
\begin{equation*}
    \frac{(k!)^2}{\prod_{i=1}^k h(v_i)},
\end{equation*}
where $v_i$'s are vertices in the bottom tree component of the associated poset and $h(v_i)$ is the number of vertices less than or equal to $v_i$.
\end{theorem}

\begin{proof}
The Hasse diagram consists of two components: a tree below and including $a$ (that happens to be a not-necessarily-full binary tree), and a tree above and including $a$ (which has height 2). The top layer of the poset has as many elements as $L$'s in the win-loss sequence, namely $k$ (recall $n=2k$). These must be the values $k+1$ through $n$, because they are greater than every other element of the poset, so there are $k!$ assignments of values to cards in the top layer. For the tree consisting of the remaining $k$ cards from card $a$ and down, we use Lemma~\ref{lem:treeposet_counting}  for a factor of
\begin{equation*}
    \frac{k!}{\prod_{i=1}^k h(v_i)}
\end{equation*}
again. Combining these two factors, we get
\begin{equation*}
    k! \cdot \frac{k!}{\prod_{i=1}^k h(v_i)} = \frac{(k!)^2}{\prod_{i=1}^k h(v_i)},
\end{equation*}
as claimed.
\end{proof}

\section{Acknowledgements}
We would like to thank the MIT PRIMES-USA program for the opportunity to conduct this research.

\printbibliography

\end{document}